%% file: arxiv_version1.tex
\documentclass[12pt, a4paper, english]{amsart}
\usepackage{amsmath} 
\usepackage{amsthm} 
\usepackage{amssymb} 

\usepackage{amscd} 
\usepackage{mathtools}
\usepackage[dvipsnames]{xcolor}

\usepackage{array} 
\usepackage[backref=page,linktocpage]{hyperref} 
\usepackage{caption} %
\usepackage{subcaption}
\usepackage{graphics,graphicx} 
\usepackage{tikz,tikz-cd} 
\usetikzlibrary{graphs,graphs.standard,calc}
\usetikzlibrary{decorations.pathreplacing,}
\usetikzlibrary{decorations.pathmorphing,shapes,decorations.markings,decorations.text}
\usetikzlibrary{automata}

\usepackage{enumerate} 
\usepackage{newtxtext,newtxmath} 

\usepackage[margin=1.25in]{geometry}

\usepackage{verbatim}

\usepackage{scalerel}

\newcommand\SetOf[2]{\left\{\left.#1\vphantom{#2}\ \right|\ #2\vphantom{#1}\right\}}
\newcommand\midSetOf[2]{\Bigl\{\Bigl.#1\vphantom{#2}\ \Bigr|\ #2\vphantom{#1}\Bigr\}}
\newcommand\smallSetOf[2]{\{{#1}\,|\,{#2}\}}

\hypersetup{
    colorlinks = true,
    linkbordercolor = {white},
    linkcolor = {BrickRed},
    anchorcolor = {black},
    citecolor = {BrickRed},
    filecolor = {cyan},
    menucolor = {red},
    runcolor = {cyan},
    urlcolor = {black}
}

\newtheoremstyle{theorems}
{13pt}
{13pt}
{\slshape}
{}
{\bfseries}
{}
{.5em}
{}

\theoremstyle{theorems}
\newtheorem{theorem}{Theorem}[section]
\newtheorem{corollary}[theorem]{Corollary}
\newtheorem{lemma}[theorem]{Lemma}

\newtheorem{proposition}[theorem]{Proposition}
\newtheorem*{theorem-no-label}{Theorem}

\newtheoremstyle{definition}
{12pt}
{12pt}
{}
{}
{\bfseries}
{}
{.5em}
{}

\theoremstyle{definition}
\newtheorem{definition}[theorem]{Definition}

\newtheorem{example}[theorem]{Example}
\newtheorem{remark}[theorem]{Remark}

\newcommand{\conv}{\operatorname{conv}}

\DeclareMathOperator{\Tutte}{T}
\newcommand{\RR}{\mathbb{R}}

\newcommand{\ZZ}{\mathbb{Z}}
\newcommand{\cT}{\mathcal{T}}
\newcommand{\cC}{\mathcal{C}}

\title{Good Triangulations of Cosmological Polytopes}
\subjclass[2020]{05A15 52B05 52B20
(primary) 05C10 05C30 (secondary)}

\keywords{Cosmological polytopes, Ehrhart theory, regular triangulations, Tutte polynomials, graph invariants}
\thanks{KF is funded by the Deutsche Forschungsgemeinschaft (DFG, German Research Foundation) under Germany´s Excellence Strategy 
– The Berlin Mathematics Research Center MATH+ (EXC-2046/1, project ID: 390685689). BS is supported by the Swedish Research Council grant 2022-04224.}

\author[A.~Benjes]{Aenne Benjes}

\address{Institute for Mathematics, Goethe University, Frankfurt am Main, Germany
}

\email{benjes@math.uni-frankfurt.de}

\author[K,.~Ferry]{Kamillo Ferry}
\address{
  Institute of Mathematics, Technische Universität Berlin, Berlin, Germany
}
\email{ferry@math.tu-berlin.de}

\author[B.~Schr\"oter]{Benjamin Schr\"oter}

\address{
  Department of Mathematics, KTH Royal Institute of Technology, Stockholm, Sweden
}
\email{schrot@kth.se}

\begin{document}
\maketitle

\begin{abstract}
Cosmological polytopes of graphs are a geometric tool in physics to study wavefunctions for cosmological models whose Feynman diagram is given by the graph. After their recent introduction by Arkani-Hamed, Benincasa and Postnikov the focus of interest shifted towards their mathematical properties, e.g., their face structure and triangulations. Juhnke, Solus and Venturello used toric geometry to show that these polytopes have a so-called good triangulation that is unimodular. 
Based on these results
Bruckamp et al.\ 
studied the Ehrhart theory of those polytopes and in particular the $h^*$-polynomials of cosmological polytopes of multitrees and multicycles.
In this article we complete this part of the story.  We enumerate all maximal simplices in good triangulations of any cosmological polytope.
Furthermore, we provide a method to turn such a triangulation into a half-open decomposition from which we deduce that the $h^*$-polynomial of a cosmological polytope is a specialization of the Tutte polynomial of the defining graph.
This settles several open questions and conjectures of Juhnke, Solus and Venturello as well as Bruckamp et al.
\end{abstract}
\thispagestyle{empty}

\section{Introduction}
\noindent
There are various polyhedral objects one may associate with a graph $G$, e.g., its (symmetric) edge polytope and their duals which are alcoved polytopes (see for example \cite{HerzogHibiOhsugi:2018, Joswig:2021, LamPostnikov:2007, LamPostnikov:2018}), the matroid polytope, independence complex or broken circuit complex of the graphical matroid of $G$ (see for example \cite{Edmonds:1970} and \cite{MatroidApplications}). In this article we study another such object, namely the cosmological polytope $\cC_G$ of the graph~$G$ which is the convex hull of the vectors $e_u+e_v-e_f$, $e_u-e_v+e_f$ and $-e_u+e_v+e_f$ for all edges $f=\{u,v\}$ of the graph~$G$; see Definition~\ref{def:cosmological_polytope} for a more detailed definition. This polytope has been introduced in \cite{ArkaniHamedBenincasaPostnikov} by Arkani-Hamed, Benincasa and Postnikov to study the physics of cosmological time evolution and
the wavefunction of the universe using positive geometries. Since then their mathematical structure has been investigated in more detail.
Benincasa \cite{Benincasa} as well as K\"uhne and Monin \cite{KuehneMonin} investigated their face structure combinatorially. Juhnke, Solus and Venturello used toric geometry and Gröbner bases to show that these polytopes poses a unimodular triangulation \cite{JuhkeSolusVenturello}, and Bruckamp, Gotermann, Juhnke, Ladin and Solus \cite{BruckampGoltermannJuhnkeLandinSolus} shifted the focus towards  the Ehrhart theory of cosmological polytopes. They found formulas for the $h^\ast$-polynomial for the families of multitrees and multicycles.
For any polytope $P$, the coefficients of this polynomial are a refinement of the (lattice) volume of the polytope $P$, and the $h^\ast$-polynomial coincides with the $h$-vector of any unimodular triangulation of $P$ whenever it exists. Moreover, it is the numerator polynomial of the generating function of the Ehrhart polynomial of $P$ which counts the lattice points in dilations of~$P$; see also Section \ref{sec:decomp} below.

In this paper we explore the Ehrhart theory and $h^\ast$-polynomial of cosmological polytopes further by considering a polyhedral and geometric approach which uses results of \cite{JuhkeSolusVenturello} and is inspired by \cite{BruckampGoltermannJuhnkeLandinSolus}, but independent of the latter.
We begin in Section~\ref{sec:good_triang} with proving Theorem~\ref{thm:main} which provides an enumeration of the maximal simplices in any so called good triangulations of a cosmological polytope in terms of decorations of the graph $G$. 

In Section~\ref{sec:decomp} we describe how one can derive a half-open decomposition from a good triangulation.
This decomposition allows us to prove with Corollary~\ref{cor:h_star} a first formula for the $h^\ast$-polynomial of a cosmological polytope. This formula has been conjectured in \cite{BruckampGoltermannJuhnkeLandinSolus} and depends on the triangulation of the polytope.

In Section~\ref{sec:h_star_polynomials} we combine the description in Theorem~\ref{thm:main} with the formula of Corollary~\ref{cor:h_star} to obtain formulas for the $h^\ast$-polynomial that are independent of the triangulation.
The formula in Theorem~\ref{thm:hstar} only requires the spanning forests of the graph $G$, and the formula in Theorem~\ref{thm:hstar2} only involves the bridge free edge sets. 
By applying these formulas to multitrees (Example~\ref{ex:multitrees}) and multicyles (Example~\ref{ex:multicycles}) we reprove the main results of \cite{BruckampGoltermannJuhnkeLandinSolus}, and are able to show that their conjectured upper bound of the coefficients of the $h^\ast$-polynomial holds true (Corollary~\ref{cor:bound}).
We further look at generalized $\Theta$-graphs (Example~\ref{ex:theta}) and some bipartite graphs (Example~\ref{ex:bipartite}) for which the $h^\ast$-polynomial or volume of the cosmological polytope has been considered before.
This way we prove yet another conjecture of Bruckamp et al.
Finally we show with Theorem~\ref{thm:tutte} that the $h^\ast$-polynomial is a specialization of the Tutte polynomial. In other words the $h^\ast$-polynomial of a cosmological polytope is a graph invariant which satisfies a deletion-contraction recurrence. This finding shows that
the coefficients of the $h^\ast$-polynomial form a ultra log-concave sequence (Corollary~\ref{cor:log-concave}), and  furthermore it allows us to express the volume of a cosmological polytope as a simple graph invariant (Corollary~\ref{cor:volume}), namely as the the number of acyclic subsets of the edges times two to the power of the number of edges. This way we 
answer the two open questions \cite[Problem 6.1 and 6.2]{JuhkeSolusVenturello}.

\section{Good triangulations}\label{sec:good_triang}
\noindent In this section we recall the definition of cosmological polytopes and good triangulations in a geometric fashion. Afterwards we prove a combinatorial bijection between certain decorated graphs and the maximal simplices in a good triangulation.

\begin{definition}\label{def:cosmological_polytope}
    Let $G=(V,E)$ be an undirected graph with $n$ nodes $V$ and $m$ edges $E$. We do allow loops, multiple edges as well as isolated nodes.
The \emph{cosmological polytope}~$\cC_G$ of the graph $G$ is the convex hull of the lattice points
\[
\SetOf{e_f, \widetilde{e}_f, \overleftarrow{e}_f, \overrightarrow{e}_f}{f\in E}\cup\SetOf{e_u}{u\in V}
\subseteq\RR^{V}\times\RR^{E}\cong\RR^{n+m}
\]
where
\[
\widetilde{e}_f = e_u+e_v-e_f,\qquad 
\overleftarrow{e}_f = e_u-e_v+e_f \qquad \text{ and } \qquad
\overrightarrow{e}_f = -e_u+e_v+e_f 
\]
for an edge $f$ which is incident to the nodes $u$ and $v$. In the remainder it might be required to distinguish the vertex $\overleftarrow{e}_f$ from $\overrightarrow{e}_f$. Therefore we may assume that the vertices of an edge are ordered whenever necessary. Moreover, we identify the space $\RR^{V}\times\RR^{E}$ with the vector space $\RR^{n+m}$.
\end{definition}
\begin{remark}
    The listed points in the definition above are all the lattice points in the cosmological polytope $\cC_G$.
    Its vertices are the points $\widetilde{e}_f$, $\overleftarrow{e}_f$, and $\overrightarrow{e}_f$ as well as the points $e_u$ for isolated nodes $u\in V$.
    We have $\overleftarrow{e}_f=\overrightarrow{e}_f=e_f$ whenever $f$ is a loop.
    Thus $\cC_G$ is a pyramid with apex $e_u$ if $u$ is an isolated node and a bipyramid if $f$ is a loop.
    Our definition agrees with the definition made in \cite{KuehneMonin}.
\end{remark}
The cosmological polytope $\cC_G$ is contained in the hyperplane $\sum x_i =1$, and contains the standard simplex with vertices $e_u$ and $e_f$ for $u\in V$ and $f\in E$. Thus it is a $n+m-1$ dimensional polytope.

We follow the ideas and names invented in \cite{JuhkeSolusVenturello} by Juhnke, Solus and Venturello to study cosmological polytopes.
They used the language of Gr\"obner bases and toric geometry to prove that each cosmological polytope has an unimodular triangulation. More precisely, any triangulation which is induced by a so called \emph{good term order} is unimodular. We follow their lead and call a triangulation~$\cT$ that uses all the lattice points in the polytope $\cC_G$ a \emph{good triangulation} if
the subdivision contains the $(n+m-1)$-dimensional standard simplex as a (maximal) cell and it is regular, i.e., induced by a height function.
In this case, the height function $\psi:\cC_G\cap\ZZ^{n+m}\to\RR$ that induces $\cT$ satisfies the following six types of \emph{fundamental inequalities}
\begin{align*}
        \psi(e_u)+\psi(e_v) &\leq \psi(e_f)+\psi(\widetilde{e}_f), &
        2\, \psi(e_f) &\leq \psi(\overrightarrow{e}_f)+\psi(\overleftarrow{e}_f)\\
        \psi(e_u)+\psi(e_f) &\leq \psi(\overleftarrow{e}_f)+\psi(e_v), &
        2\, \psi(e_v) &\leq \psi(\widetilde{e}_f)+\psi(\overrightarrow{e}_f)\\
        \psi(e_v)+\psi(e_f) &\leq \psi(\overrightarrow{e}_f)+\psi(e_u), &
        2\, \psi(e_u) &\leq \psi(\widetilde{e}_f)+\psi(\overleftarrow{e}_f)
\end{align*}
for every edge $f$ incident to the nodes $u$ and $v$. Furthermore, the following inequalities are satisfied
\[
\sum_{f\in C} \psi(e_f)\leq \sum_{f\in C} \psi(\overleftarrow{e}_f) \text{ and } \sum_{f\in C} \psi(e_f)\leq \sum_{f\in C} \psi(\overrightarrow{e}_f)
\]
for every (oriented) cycle $C=\{(u_1,u_2),(u_2,u_3)\ldots,(u_{k-1},u_k),(u_k,u_1)\}\subseteq E$ of $G$.

Examples of good triangulations are all placing triangulations of the lattice points in the cosmological polytope $\cC_G$ which place the standard simplex first, or regular triangulations whose lifting function is zero on the standard simplex and (strictly) positive on all other points.
In this paper we view a triangulation as a simplicial complex and identify the vertices of a simplex in the triangulation with the simplex itself. For further definitions, notions and a more detailed background on triangulations we point the reader to \cite{DeLoeraRambauSantos:2010}.

Our aim is to count (half-open) simplices in a good triangulations $\cT$. Each maximal simplex $S\in\mathcal{T}$ gives rise to a \emph{decoration} of $G$ that we denote by $G_S$.
A \emph{decoration} consists of two types of nodes, \emph{selected nodes} $V(S) = \smallSetOf{u\in V}{e_u\in S}$ and unselected nodes $V\setminus V(S)$, the edges might be \emph{squiggly} $\widetilde{E}(S)=\smallSetOf{f\in E}{\widetilde{e}_f\in S}$, \emph{directed} $\overleftarrow{E}(S)\cup\overrightarrow{E}(S)$ where $\overrightarrow{E}(S) = \smallSetOf{f\in E }{\overrightarrow{e}_f\in S}$ and $\overleftarrow{E}(S) = \smallSetOf{f\in E }{\overleftarrow{e}_f\in S}$ or \emph{selected} $\widehat E(S) = \smallSetOf{f\in E}{e_f\in S}$; see also \cite{JuhkeSolusVenturello} and \cite{BruckampGoltermannJuhnkeLandinSolus}.
We call an edge a \emph{double edge} if it is selected and directed in one of the two ways. 
We denote the set of all double edges by $D(S)= (\overleftarrow{E}(S)\cup\overrightarrow{E}(S))\cap \widehat{E}(S)$.
Moreover, we say an edge is a \emph{simply decorated edge} if it is neither squiggly nor an double edge, that is it is directed or selected but not both.
It is worth mentioning that a squiggly edge can never be directed or selected in $G_S$.

In our pictures we mark selected nodes in $G_S$ as filled circles and we draw edges $f\in E$ according to their decoration, it is present whenever it is a selected edge which we denote by $f$, squiggly $\widetilde{f}$ if $\widetilde{e}_f\in S$ or directed in one of two ways $\overleftarrow{f}$ and $\overrightarrow{f}$. See Figure~\ref{fig:ex1} and Figure~\ref{fig:ex2} further below for examples.\footnote{Our drawings and labeling of decorated graphs does not agree with those used in \cite{JuhkeSolusVenturello} and \cite{BruckampGoltermannJuhnkeLandinSolus}, but is more natural to us.}

\begin{figure}
   \begin{subfigure}[c]{0.49\textwidth}
   \includegraphics{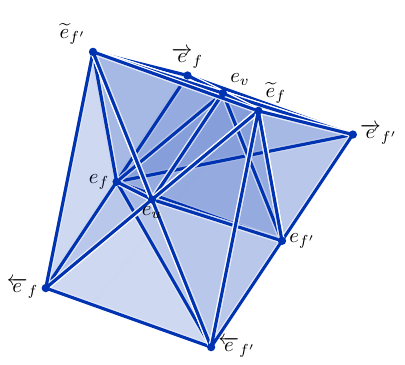}
      \subcaption{The triangulated polytope.\label{subfig:triang}}
   \end{subfigure}
   \begin{subfigure}[c]{0.49\textwidth}
   \input{example1dec.tikz}
      \subcaption{The twelve decorated graphs of $\cT$.\label{subfig:dec}}
   \end{subfigure}
   \caption{A good triangulation $\cT$ of the cosmological polytope of Example~\ref{ex:ex1}.\label{fig:ex1}}
\end{figure}

\begin{example}\label{ex:ex1}
    Consider the graph $G$ on two nodes with two parallel edges. Its cosmological polytope $\cC_G$ is a prism over an triangle. A good triangulation of this three-dimensional polytope and all its ten lattice points are depicted in Figure~\ref{subfig:triang}).
    The twelve decorations of the graph $G$ that are in bijection to the twelve maximal simplices in this triangulation are shown in Figure~\ref{subfig:dec}).
    In total the polytope $\cC_G$ has four good triangulations depending whether $\psi(\overrightarrow{e}_f)+\psi(e_{f'})$ is larger than $\psi(\overrightarrow{e}_{f'})+\psi(e_f)$ and whether $\psi(\overleftarrow{e}_f)+\psi(e_{f'})$ is larger than $\psi(\overleftarrow{e}_{f'})+\psi(e_f)$ or not.
\end{example}

The following is a key lemma in our approach which allows us to delete edges or swap simply decorated edges for squiggly edges.
\begin{lemma}\label{lem:tool}
    Let $G$ be the subgraph of the graph $\hat G$ on the same set of nodes where the edge $f$ with nodes $u$, $v$ is removed.
    Furthermore let $\cT_{\hat G}$ be a good triangulation of the polytope $\cC_{\hat G}$ which is induced by the height function $\psi$, and $\cT_G$ be the triangulation of $\cC_G$ that is induced by restricting $\psi$ to the lattice points in $\cC_G$. 
    Then $S$ is a simplex in the triangulation $\cT_G$ of $\cC_G$ if and only if
    $S\cup\{\widetilde{e}_f\}$ is a simplex in the triangulation $\cT_{\hat G}$.
    Furthermore, this is the case if and only if $S\cup\{e_f\}$, $S\cup\{\overrightarrow{e}_f\}$ or $S\cup\{\overleftarrow{e}_f\}$ is a simplex in the triangulation $\cT_{\hat G}$ of $\cC_{\hat G}$.
\end{lemma}
\begin{proof}
    First we observe that the intersection of the cosmological polytope $\cC_{\hat G}$ with the hyperplane $H=\smallSetOf{x\in\RR^{n+m}}{x_f = 0}$ is an embedding of the  polytope $\cC_G$ that we may identify with this cosmological polytope.
    This hyperplane $H$ separates $\widetilde{e}_f$ from the three points $\overleftarrow{e}_f$, $e_f$ and $\overrightarrow{e}_f$ while all other points lie on this hyperplane.
    We deduce that the lower-dimensional simplex $S'\in\cT_G$ which lies in $H$ is a face of the simplex $S'\cup\{\widetilde{e}_f\}$ and at least one simplex of the form $S'\cup\{p\}$ where $p\in\{
    e_f, \overleftarrow{e}_f, \overrightarrow{e}_f\}$. 
    
    Furthermore, the fundamental inequalities imply that no simplex $S$ in $\cT_{\hat G}$ contains any of the following three subsets of lattice points 
    $\{\widetilde{e}_f, e_f\},\, \{\widetilde{e}_f, \overleftarrow{e}_f\},\, \{\widetilde{e}_f, \overrightarrow{e}_f\}$.
    Thus the simplex $S$ lies entirely in $H$ or on one side of $H$. 
    We conclude that if a simplex $S\in \cT_{\hat G}$ includes exactly one vertex $p$ outside of $H$ then its face $S\setminus\{p\}$ is a simplex in $\cT_G$ which completes the proof.
\end{proof}

Note that for a maximal simplex $S$ in $\cT_G$ only one of the three sets  $S\cup\{e_f\}$, $S\cup\{\overrightarrow{e}_f\}$ or $S\cup\{\overleftarrow{e}_f\}$ forms a simplex in $\cT_{\hat G}$ as this simplex shares the facet $S$ with the maximal simplex $S\cup\{\widetilde{e}_f\}$.

Before we continue our journey of finding a way to enumerate all cells in good triangulations, we formulate a few useful statements and observations. We begin by considering the affine coordinates with respect to a maximal simplex.

\begin{lemma}\label{lem:coord}
    Let $S$ be a maximal simplex in $\cT$ and $w\in V$. Then there are unique numbers 
    $\lambda^w_u$ and $\lambda^w_f$ for $u\in V(S)$ and $f\in D(S)$ such that
    \begin{equation}\label{eqn:affine-coordinates}
        e_w \ = \ \sum_{u\in V(S)} \lambda^w_u e_u + \sum_{f=\{u,v\}\in D(S)} \lambda^w_f (e_u-e_v) \enspace.
    \end{equation}
\end{lemma}
\begin{proof}
    The $n+m$ vertices of the maximal simplex $S$ form an affine basis for the hyperplane $\smallSetOf{x\in\RR^{n+m}}{\sum x_i = 1}$ which contains the polytope $\cC_G$. Thus every point in $\RR^{n+m}$, and in particular $e_w$, can be expressed in a unique way as a linear combination of these vertices. That is $e_w = \sum_{p\in S} \lambda_p\, p$ for some $\lambda_p\in\RR$.

    For every edge $f\in E$, the $e_f$-coordinate of $e_w$ vanishes, thus $\lambda_p = 0$ whenever $p\in\{\widetilde{e}_f,\overleftarrow{e}_f, \overrightarrow{e}_f, e_f\}$ and $f\not\in D(S)$. Furthermore, for the same reason, if $f\in D(S)$ then $\lambda_{q}-\lambda_{e_f}=0$ for $q\in\{\overleftarrow{e}_f,\overrightarrow{e}_f\}\cap S$. As $\overleftarrow{e}_f-e_f = e_f-\overrightarrow{e}_f = e_u-e_v$ for $f=(u,v)$ we conclude that the linear combination has the desired form \eqref{eqn:affine-coordinates}.
\end{proof}

A consequence of the previous lemma is the following observation.
\begin{corollary}\label{cor:selecetd_nodes}
    Let $S$ be a maximal simplex in $\cT$, and $C$ be a connected component of the graph $G|_{D(S)}$ with vertices $V$ and edges $D(S)$.
    Then the decorated graph $G_S$ restricted to the component $C$ includes a unique selected node. 
\end{corollary}
\begin{proof}
    Suppose $w$ and $w'$ are two selected nodes in $G_S$, that is $w,w'\in V(S)$, and they lie in the same connected component, that is there is a path of double edges from $w$ to $w'$. This would result in a dependence in 
    \eqref{eqn:affine-coordinates} contradicting Lemma~\ref{lem:coord}. 
\end{proof}

\begin{remark} Alternatively Corollary~\ref{cor:selecetd_nodes} can be read off of the zig-zag binomials in the Gr\"obner basis the authors constructed in \cite{JuhkeSolusVenturello}.

Furthermore, it might be worth mentioning that
 the coefficients in \eqref{eqn:affine-coordinates} are elements in $\{-1,0,1\}$ which follows from Corollary~\ref{cor:selecetd_nodes} as they record if the path from $w$ to the unique selected node $v$ passes through the double edges.
 We also see that $\lambda_u^w= 1$ if $u\in V(S)$ is connected to $w$ via double edges and $\lambda_u^w=0$ otherwise. 
\end{remark}

Our next statements help us to understand the maximal simplices in a good triangulation. They are analogous to
 \cite[Lemma 4.5]{BruckampGoltermannJuhnkeLandinSolus} but they include all good triangulations and the lemma does not make any assumptions on the cosmological polytope.
\begin{lemma}\label{lem:cycle}
    Let $S$ be a simplex in the good triangulation $\cT$ of the polytope $\cC_G$, then the decorated graph $G_S$ does not contain a cycle of double edges.
\end{lemma}
\begin{proof}
    We show this statement by contradiction.
    For this purpose suppose that $S$ is a simplex and $C = \{(u_1,u_2),\ldots,(u_{k-1},u_k),(u_k,u_1)\}\subseteq E$ is the underlying cycle of edges in $G$ for a cycle of double edges in the decorated graph $G_S$.
    The edges in $C$ are partitioned into the two sets
    \[
        C_1 = \SetOf{f\in C}{\overleftarrow{e}_f\in S} \text{ and }
        C_2 = \SetOf{f\in C}{\overrightarrow{e}_f\in S} \enspace .
    \]
    The point
    \[
        x\ = \ \frac{1}{|C|} \sum_{f\in C_1} \overleftarrow{e}_{f} + \frac{1}{|C|} \sum_{f\in C_2} e_f \ = \ \frac{1}{|C|} \sum_{f\in C_1} e_{f} + \frac{1}{|C|} \sum_{f\in C_2}\overrightarrow{e}_{f} 
    \]
    is the convex combination of two distinct sets of vertices in $S$. Thus this point is a witness that $S$ cannot be a simplex, contradicting our assumption.
\end{proof}

Before we discuss the general case let us consider the cosmological polytope of a forest.

\begin{proposition}\label{prop:facets_trees} Let $G=(V,E)$ be a forest and $\cT$ a good triangulation of the cosmological polytope $\cC_G$. Then for any choice of squiggly and double edges (with orientations) there exists a unique maximal simplex $S$ in $\cT$ for which $G_S$ has exactly the chosen decorated edges.
\end{proposition}
\begin{proof}
    We claim that it is enough to prove that for any partition $E_1 \cup E_2 \cup E_3$ of  the edge set~$E$, there is a unique maximal simplex $S\in\cT$ with $\widetilde{E}(S) = E_1$, $\overleftarrow{E}(S) = E_2$ and $\overrightarrow{E}(S)=E_3$.
    This is because by Lemma~\ref{lem:tool}, we are allowed to swap any simply decorated edge of $G_S$ to a squiggly edge, whereby the simplex remains maximal.
    
    Thus let $E_1\cup E_2\cup E_3=E$ be a such a partition and $D=E_2\cup E_3$. As the graph~$G$ is cycle free and we have selected no vertex yet there is exactly one simplex $S' \in \cT$ with $|E_1|+2|D|$ many elements for which $\widetilde{E}(S') = E_1$, $\overleftarrow{E}(S') = E_2$ and $\overrightarrow{E}(S')=E_3$.  
    We point out that the simplex $S'$ lies in the face 
    \begin{align*}
            F\ &=\ \cC_G \cap \midSetOf{x\in\RR^{n+m}}{\sum_{f\in D} x_f - \sum_{f\in E_1} x_f=1}\\
            &=\ \conv(\SetOf{\widetilde{e}_f}{f \in E_1}\cup\SetOf{\overrightarrow{e}_f, \overleftarrow{e}_f}{f\in D}
    \end{align*}
    of the cosmological polytope $\cC_G$. 
    The face $F$ can be constructed by taking $|D|$ many bipyramids over the simplex $\widehat{S}=\smallSetOf{e_f}{f\in D}$ with apices $\overleftarrow{e}_f$ and $\overrightarrow{e}_f$, and afterwards taking multiple pyramids with apices $\widetilde{e}_f$ for each $f\in E_1$.

    Since the good triangulation $\cT$ is unimodular we conclude that $\cT$ restricted to $F$ consists of $2^{|D|}$ cells of dimension $|E_1|+2|D|$ all of which do meet in $\widehat{S}$. These cells are determined by the choice of which of the two points $\overleftarrow{e}_f$ and $\overrightarrow{e}_f$ is in that cell.
    One of these cells is the simplex $S'$ which therefore is a cell in $\cT$.
    The cell $S'$ is contained in some maximal simplex $S\in\cT$.
    
    As for any edge $f \in E$ either $\widetilde{e}_f \in S'$, $\{\overrightarrow{e}_f, e_f\} \subseteq S'$ or $\{\overleftarrow{e}_f,e_f\} \subseteq S'$, the fundamental inequalities imply that no more edges can be added and thus
    \[S = S' \cup \SetOf{e_v}{v \in V'} \text{ for some } V' \subseteq V.\]
    We are left with the task to show uniqueness of this maximal simplex~$S$. It follows from Corollary~\ref{cor:selecetd_nodes} that if we consider the subgraph of $G_S$ induced by the double edges $D(S)$, then every component has a unique vertex $v$ such that $e_v \in S$. 
    Thus there is exactly one maximal simplex $S$ with the described decorated edges.
\end{proof}

Now we are prepared to extend the statement of Proposition~\ref{prop:facets_trees} to all graphs.
\begin{theorem}\label{thm:main} Let $\cT$ be a good triangulation of $\cC_G$.
    For any choice of squiggly and cycle free double edges there exists a unique maximal simplex $S\in\cT$ for which the graph~$G_S$ has the chosen decorated edges. Moreover, there are no other maximal simplices in~$\cT$.
\end{theorem}
\begin{proof}
Suppose $G$ is a graph with the smallest number of edges such that the claim of the statement is false.
That is either there is no simplex $S$ corresponding to a given decoration of edges or there are at least two simplices $S$ and $S'$ with the same edge decoration. In any of those cases, $G$ must contain a cycle $C$ by Proposition~\ref{prop:facets_trees}.
We deal first with the assumption that for a given edge decoration there is no simplex $S$ of which $G_S$ has the described edge decoration. But one of the edges, say $f$, in $C$ must be a squiggly or a simply decorated edge by Lemma~\ref{lem:cycle}, and thus Lemma~\ref{lem:tool} applies. If we remove $f$ from $G$ we find a simplex with the correct edge decoration on the remaining edges which we can complete to a maximal simplex in $\cT$ with the required decoration by applying Lemma~\ref{lem:tool} once more. This is a contradiction to our claim thus such a simplex must exist.
Similarly, if two simplices $S$ and $S'$ in $\cT$ have the same edge decoration, then in both of them there must be an edge $f\in C$ that is not a double edge. We can apply Lemma~\ref{lem:tool} once more and delete that edge to obtain a contradiction as $S\setminus\{e_f, \widetilde{e}_f, \overleftarrow{e}_f, \overrightarrow{e}_f\} = S'\setminus\{e_f, \widetilde{e}_f, \overleftarrow{e}_f, \overrightarrow{e}_f\}$ and thus $S=S'$.
Finally Lemma~\ref{lem:cycle} shows that there are no other simplices in $\cT$ than those described in the statement.
\end{proof}

\section{A decomposition into half-open simplices}\label{sec:decomp}
\noindent In this section we aim to construct from a good triangulation $\cT$ of the cosmological polytope $\cC_G$ a half-open decomposition to draw conclusions about the
$h^\ast$-polynomial of the cosmological polytope which agrees with the face-counting $h$-polynomial of the unimodular triangulation $\cT$. 

A set $Q$ is called a \emph{half-open polytope} if the euclidean closure $\overline{Q}$ is a polytope and
\[
   Q \ = \ \overline{Q}\setminus \bigcup_{F\in B} F
\]
for some family $B$ of facets of $\overline{Q}$.
A \emph{half-open decomposition} of a polytope $P$ is a collection of pairwise disjoint half-open polytopes that cover the polytope $P$.

To construct the desired half-open decomposition, we first have to analyze the \emph{dual graph} of the triangulation $\cT$, i.e., the graph whose nodes are the maximal simplices ${S\in\cT}$ and two nodes are connected by an edge whenever the simplices share a codimension-$1$ face.

The following slightly more technical definition is the main tool to construct the desired half-open decomposition.
\begin{definition}\label{def:anchored_path}
   We call a path $S_1,\ldots, S_k$ with at least two nodes 
   in the dual graph of $\cT$ \emph{anchored} if 
   \begin{enumerate}[i)]
   	\item $S_1$ has a smaller number of squiggly and double edges than $S_2$,
	\item\label{it:anchored} for every $1 < i < k$ there is an edge $f\in E$ such that $e_f, p\in S_i$ for $p\in\{\overrightarrow{e}_f, \overleftarrow{e}_f\}$, as well as 
$S_i\setminus \{p\}$ and $S_i\setminus \{e_f\}$ are each a face of one of the two simplices $S_{i-1}$ and $S_{i+1}$.
   \end{enumerate}
\end{definition}

\begin{figure}
   \begin{subfigure}[c]{0.49\textwidth}
   \input{example2dec.tikz}
      \subcaption{The sixteen decorated graphs of $\cT$.\label{subfig:dec2}}    
   \end{subfigure}
   \begin{subfigure}[c]{0.49\textwidth}
   \input{example2dualgraph.tikz}
      \subcaption{The (oriented) dual graph of $\cT$.\label{subfig:dual_graph}}
   \end{subfigure}
   \caption{The good triangulation $\cT$ and the dual graph of Example~\ref{ex:ex2}.\label{fig:ex2}}
\end{figure}
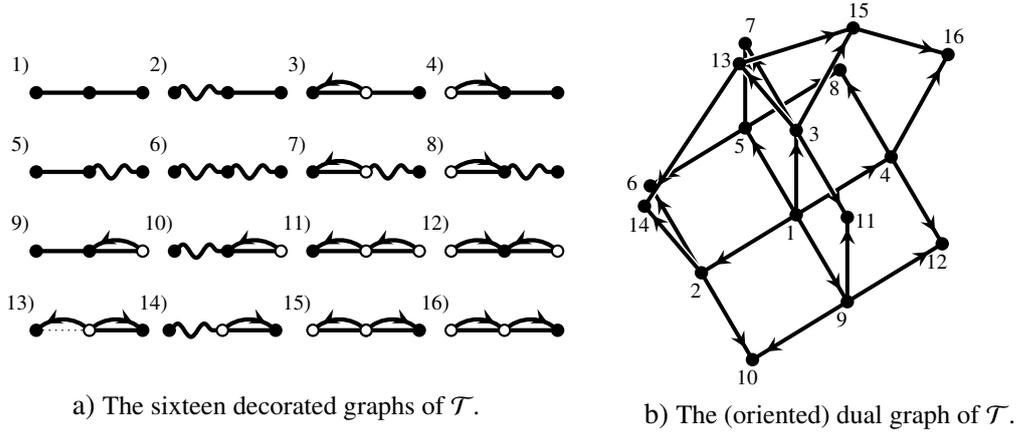
 \begin{example}\label{ex:ex2}
    Consider the connected graph $G$ on the three nodes $u$, $v$, $w$ with two edges $f=(u,v)$ and $f'=(v,w)$.
    We consider the good triangulation for which Figure~\ref{subfig:dec2}) shows all sixteen decorated graphs of maximal simplices.
    For example the pair $S_9$, $S_{10}$ or $S_{13}$, $S_{15}$ are anchored paths as well as 
    $S_1, S_3, S_{13}$ and $S_{13}, S_{15}, S_{16}$, but $S_{15}, S_{16}$ is not an anchored path. Figure~\ref{subfig:dual_graph}) shows the dual graph oriented along the anchored paths.
\end{example}

\begin{remark}
    It is worth  mentioning that the chosen order of maximal simplices is a shelling order, but not the linear shelling order which is induced by the height function~$\psi$ that induces the triangulation. To obtain this shelling order for Example~\ref{ex:ex2} one has to swap the order of simplex $S_{13}$ and $S_{14}$.
\end{remark}

The next step is to construct half-open simplices from the anchored paths. For a maximal simplex $S\in\cT$ let us consider all anchored paths ending in $S$ and the set 
\[
     B(S) \ =\ \SetOf{S_{k-1}\cap S}{S_1,\ldots,S_k \text{ is an anchored path with $S_k=S$}} \enspace .
\]
Now we denote by $\Delta(S)$ the half-open simplex
\[
\Delta(S) \ = \ \conv(S) \setminus \bigcup_{F\in B(S)} F \enspace .
\]

\begin{lemma}\label{lem:num_anchored_paths} Let $S\in\cT$ be maximal. For every squiggly and double edge in $G_S$ there exists an anchored path ending in $S$.
\end{lemma} 
\begin{proof}
    First let us assume $\widetilde{e}_f\in S$. Let $S'$ be the maximal simplex for which $\widetilde{e}_f\not\in S'$ and the squiggly and double edges of $G_{S'}$ other than $f$ agree with those of $G_S$. Then by Lemma~\ref{lem:tool} the path $S', S$ is anchored.
    
    Now let us assume $e_f, \overrightarrow{e}_f\in S$. Let $S'$ be the maximal simplex with fewer double edges and for which the squiggly and oriented double edges of $G_{S'}$ other than $f$ agree with those of $G_S$.
    By Corollary~\ref{cor:selecetd_nodes} there is a unique node $w$ such that $e_w\in S'\setminus S$.
    Moreover, by comparing the affine coordinates of $e_w$ with respect to $S$ and $S'$ we deduce that the coefficient $\lambda^w_f$ of $\overrightarrow{e}_f-e_f$ in the former combination does not vanish.
    We conclude that the point $e_w$ is either separated from $e_f$ by the hyperplane spanned by $S\setminus\{e_f\}$, or $e_w$ is 
    separated from $\overrightarrow{e}_f$ by the hyperplane spanned by $S\setminus\{\overrightarrow{e}_f\}$.
    In any of these two cases there is a simplex $S''$ with $f\not\in D(S'')$ neighboring $S$ along the facet $S\setminus\{e_f\}$ or $S\setminus\{\overrightarrow{e}_f\}$.
    We see that $S''=S'$ whenever $G_{S''}$ has fewer double edges than $S$. Thus $S', S$ is an anchored path in this case. Otherwise $S''$ has the same number of double edges as $S$ and hence some (unique) element $f'\in \overleftarrow{E}(S'')\setminus \overleftarrow{E}(S)\cup \overrightarrow{E}(S'')\setminus \overrightarrow{E}(S)$.
    Therefore we may apply the same arguments to $S''$ and the double edge $f'$ while $S'$ and $w$ remain unchanged.
    This procedure continues till we have either found the anchored path $S_1=S',\ldots, S''=S_{k-1}, S_k=S$ or a cycle in which all maximal simplices satisfy condition \ref{it:anchored}) of anchored paths in Definition~\ref{def:anchored_path}.
    We want to show that the latter situation can never occur. Thus suppose we have found such a cycle $S_1,\ldots S_k$ with $S_1=S_k=S$.
    We know that $k>3$ as $S$ and $S''$ share only one facet and $e_w$ lies on the same side as $S''$ of that facet defining hyperplane.
    The simplex $S_1=S_k=S$ is the only one that contains both points $\overrightarrow{e}_f$ and $e_f$, but either $e_f\in S\cap S_2$ and $\overrightarrow{e}_f\in S\cap S_{k-1}$, or  $\overrightarrow{e}_f\in S\cap S_{2}$ and $e_f\in S\cap S_{k-1}$. This is a contradiction because
    each simplex in the sequence contains at least one of the three points $\overrightarrow{e}_f$, $\overleftarrow{e}_f$ and $e_f$ and besides $S$ there is at most one other simplex that contains two of them which would be $\overleftarrow{e}_f$ and $e_f$, but these simplices are the only ones which allow for a change of the decoration.
    
    The situation for $e_f, \overleftarrow{e}_f\in S$ is analogous and thus we have found all the desired anchored paths.
\end{proof}

\begin{lemma}\label{lem:disjoint}
    Let $S$ and $S'$ be two distinct maximal simplices in $\cT$, then $\Delta(S)\cap\Delta(S') = \emptyset$.
\end{lemma}
\begin{proof}
    We prove this by looking at the difference $S\setminus S'$.
    If $\widetilde{e}_f\in S\setminus S'$, then there is an anchored path $S_1, S_2$ ending in $S_2=S$ such that $S_1\cap S_2= S\setminus\{\widetilde{e}_f\}$, hence 
    \[
    S\cap S'\subseteq S\setminus\{\widetilde{e}_f\} \subseteq B(S).
    \]
    If $e_f,p\in S$ for $p\in\{\overrightarrow{e}_f, \overleftarrow{e}_f\}$ and either $e_f\not\in S'$ or $p\not\in S'$ there must be again an anchored path ending in $S$. This time passing through a facet $F$ of the simplex $S$ which is either $S\setminus\{p\}$ or $S\setminus\{e_f\}$. We see again that $S\cap S'\subseteq B(S)$. Therefore $\Delta(S)$ does not intersect the simplex $\conv S'$ in this case.
    If $\Delta(S)\cap\Delta(S')\neq \emptyset$ we may conclude by symmetry that both decorated graphs $G_{S\setminus S'}$ and $G_{S'\setminus S}$ do not have any squiggly or double edges, but that means that the squiggly or double edges of $G_S$ and $G_{S'}$ agree and hence $S=S'$.
\end{proof}

The anchored paths $S_1,\ldots,S_{k-1}, S_k$ impose an orientation on the edges of the dual graph of $\cT$ by directing edges from the simplex $S_{k-1}$ to $S_k$. Furthermore, Lemma~\ref{lem:disjoint} shows that all edges are indeed in some anchored path. 

\begin{lemma}\label{lem:dual_graph_acyclic}
    The orientation of the dual graph of $\cT$ induced by the anchored paths is acyclic.
\end{lemma}
\begin{proof}
    Suppose there is an oriented cycle $K_1,\ldots,K_\ell, K_{\ell+1}=K_1$ of maximal simplices in the dual graph of the triangulation $\cT$.
    Clearly these simplices all have the same number of squiggly as well as double edges because there is no edge that is oriented from $S$ to $S'$ if $S'$ has fewer squiggly and double edges than $S$.
    To show that there is no oriented cycle consider
    the point
    \[
        q \ =\ \frac{1}{n}\sum_{w\in V} e_w 
    \]
    and an anchored path $S_1,\ldots,S_k$ such that $S_{i-1}$ and $S_{i}$ are two neighboring simplices $S_{i-1}=K_{i'-1}$ and  $S_{i}=K_{i'}$ in the cycle for some index $i'$.
    Furthermore, let $w'$ be the unique node that is selected in $S_1$ but not in $S_{i-1}$ or $S_{i}$ which exist by Corollary~\ref{cor:selecetd_nodes}.
    There is also an edge $f'\in \overleftarrow{E}(S_{i})\setminus \overleftarrow{E}(S_{i-1})\cup \overrightarrow{E}(S_{i})\setminus \overrightarrow{E}(S_{i-1})$ and a node~$v'$ that is selected in $S_{i}$ and is connected to $f'$ by double edges.
    By Lemma~\ref{lem:coord} we may express the point $q$ as the following linear combination of the vertices in $S_i$:
    \begin{align*}\label{eqn:affine-coordinates}
        q  &= \frac{1}{n} \sum_{w\in V} \Biggl( \sum_{u\in V(S_i)} \lambda^w_u e_u + \sum_{f=(u,v)\in D(S_i)} \lambda^w_f (e_u-e_v) \Biggr)\\
        & = \frac{1}{n} \sum_{u\in V(S_i)} \sum_{w\in V}\lambda^w_u e_u 
        + \frac{1}{n}\sum_{f=(u,v)\in D(S_i)} \sum_{w\in V\setminus V(S_i)}
        \lambda^w_f (e_u-e_v)
        \enspace.
    \end{align*}
    Here we may order the endpoints $u$ and $v$ of the edge $f=(u,v)$ such that $e_u-e_v = p-e_f$ for  $p\in\{\overleftarrow{e},\overrightarrow{e}\}\cap S_i$.
    The number $\lambda^{w'}_{f'}\neq 0$. Furthermore for all other nodes $w\in V$ holds either $\lambda^{w}_{f'}=0$, if the path of double edges from $w$ to a selected node in $G_{S_i}$ does not include $f'$, or $\lambda^{w}_{f'}=\lambda^{w'}_{f'}$, if the path to $v'$ uses the edge $f'$ as the path to $v'$ passes through the edge $f'$ in the same direction.    
    We conclude that the facet that separates $S_i$ from $S_{i-1}$ also separates $S_i$ strictly from $e_{w'}$ and $q$.
    A direct consequence of these separating hyperplanes is that the oriented dual graph of $\cT$ has no oriented cycle. If there would be an oriented cycle, then the intersection of separating half-spaces that contain $S_{i-1}$ and not $S_i$ would be lower dimensional, but a neighbourhood of $q$ is fully contained in this intersection.   
\end{proof}

Now we show that the half-open simplices in our construction form indeed a half-open decomposition. 

\begin{proposition}\label{prop:decomp}
    The collection
    $
    \SetOf{\Delta(S)\subseteq\RR^{n+m}}{S\in\cT \text{ is maximal}}
    $
    forms a half-open decomposition of the cosmological polytope $\cC_G$.
\end{proposition}
\begin{proof}
By Lemma~\ref{lem:disjoint}, the half-open simplices \(\Delta(S)\) are disjoint.
Thus, we only have to show that for every simplex $S'\in\cT$ the relative interior of $\conv S'$ is covered.

Consider the set $\mathcal{S}$ of all maximal simplices that contain $S'$. This set is non-empty as $S'$ is a face of \(\cT\). 
The anchored paths induce an acyclic orientation on the dual graph of $\cT$ by Lemma~\ref{lem:dual_graph_acyclic} .
This gives a partial order on $\mathcal{S}$. Let $S$ be a minimal element in $\mathcal{S}$ with respect to that ordering.

If the relative interior of \(S'\) is already contained in \(\Delta(S)\), there is nothing to show.
Thus, suppose the relative interior of \(S'\) is not contained in \(\Delta(S)\). Then, there is an 
anchored path $S_1,\ldots, S_{k-1},S_k$ ending in $S_k=S$ such that $F\subseteq S_{k-1}\cap S$. 
In this case $S_{k-1}\in\mathcal{S}$ and $S_{k-1}$ is a smaller element with respect to the partial ordering 
induced by the anchored paths. This contradicts the minimality of the simplex $S$.
Thus the relative interior of $S'$ is covered by the half-open simplex $\Delta(S)$. 
We conclude that all relatively open faces are covered, and therefore the collection of all these half-open simplices forms a half-open decomposition of the polytope $\cC_G$.
\end{proof}

\begin{remark}
    The proof of Lemma~\ref{lem:dual_graph_acyclic} indicates that a perturbation of the point $q$ into the interior of the standard simplex can be used as a point of visibility for the half-open decomposition. However, our more combinatorial approach allows us to connect the removed faces a simplex $S$ to the squiggly and double edges of the decorated graph $G_S$.
\end{remark}

The unimodular half-open decomposition of Proposition~\ref{prop:decomp} allows us naturally to draw conclusions about the $h^\ast$-polynomial of a cosmological polytope. At the center of Ehrhart theory stands the function $j\mapsto|j\,P\cap \ZZ^n|$ which counts the number of lattice points in the $j$-th dilation of a lattice polytope $P\subseteq\RR^n$, i.e., a polytope whose vertices all have integral coordinates. It is a famous result of Eugène Ehrhart \cite{Ehrhart1962} that this function is polynomial, and thus is called the \emph{Ehrhart polynomial}. One may express the generating function of the Ehrhart polynomial of a $d$-dimensional polytope $P$ as rational function, that is
\[
    \sum_{j=0}^\infty |j\,P\cap \ZZ^n|\,z^j \ = \ \frac{h^\ast(z)}{(1-z)^{d+1}} 
\]
where the expression $h^\ast(z)$ is the \emph{$h^\ast$-polynomial} of $P$. The $h^\ast$-polynomial is a polynomial of degree at most $d$ with non-negative coefficients. For further details of Ehrhart theory and $h^\ast$-polynomials we point the reader to the two books \cite{BeckRobins:2015} and \cite{BeckSanyal:2018}. In this article we just need the well known result that the $h^\ast$-polynomial of the half-open simplex $\Delta(S)$ is $z^k$ if $G_S$ has $k$ squiggly and double edges as we removed exactly $k$ of the facets from the maximal simplex $S$. For a proof or further explanations of this fact see for example \cite[Theorem 5.5.3]{BeckSanyal:2018}.

We derive the following statement which was conjectured by Bruckamp, Goltermann, Juhnke, Landin, and Solus in \cite[Conjecture 4.15]{BruckampGoltermannJuhnkeLandinSolus}.
\begin{corollary}\label{cor:h_star}
Let $\cT$ be a good triangulation of $\cC_G$.
The $h^\ast$-polynomial of the cosmological polytope $\cC_G$ is
\[
    h^\ast(\cC_G; z) \ = \ \sum_{S\in\cT} z^{k(S)}
\]
where $k(S) = |\widetilde{E}(S)|+|D(S)|$ denotes the number of squiggly and double edges in $G_S$.
\end{corollary}

\section{$h^\ast$-polynomials of cosmological polytopes}\label{sec:h_star_polynomials}
\noindent In this section we going to improve our formula for the $h^*$-polynomial of cosmological polytopes such that we do not require a triangulation and make use of Theorem~\ref{thm:main}. Moreover, we present several examples which connect our findings to previous results and conjectures.

The starting point in this section is the following formula.
\begin{theorem}\label{thm:hstar}
    Let $G$ be a multigraph with $m$ edges and loops, then
    \[
    h^\ast(\cC_G; z) \ = \ \sum (2z)^{|H|} (1+z)^{m-|H|}
    \]
    where the sum is taken over all subsets $H$ of $E$ such that the induced subgraph $G|_H$ with nodes $V$ and edges $H$ is acyclic.
\end{theorem}
\begin{proof}
    Fix a good triangulation $\cT$ of the polytope $\cC_G$ and a subset $H\subseteq E$ of edges of the graph $G$.
    Now consider all simplices $S$ in $\cT_G$ for which the double edges in $G_S$ agree with $H$.
    By Lemma~\ref{lem:cycle} there is no such simplex whenever $H$ contains a cycle, thus we restrict ourselves to those sets $H$ that are acyclic.
    In this case $G_S$ has precisely $|H|$ double edges and $|E\setminus H|$ edges that are either simply decorated or squiggly. In total we consider $2^m$ simplices that we denote by $\cT(H)$. Furthermore, there are $\tbinom{m-|H|}{\ell}$ ways to select $\ell$ squiggly edges from $E\setminus H$ and $2^{|H|}$ ways to select the orientation of the double edges.
    In each of those cases the number of squiggly and double edges is $\ell+|H|$ and thus Corollary~\ref{cor:h_star} leads to
    \[
    \sum_{S\in\cT(H)} z^{k(S)} \ = \ \sum_{\ell=0}^{m-|H|} 2^{|H|}\binom{m-|H|}{\ell} \, z^{\ell+|H|} \ = \ (2z)^{|H|} (1+z)^{m-|H|}
    \]
    when we restrict ourselves to the simplices in $\cT(H)$. Now varying $H$ leads to the claimed formula as the sets $\cT(H)$ are disjoint.
\end{proof}

One of the many consequences of Theorem~\ref{thm:hstar} is that the $h^*$-polynomials of cosmological polytopes satisfy the following deletion-contraction recurrence.
\begin{corollary}\label{cor:recurrence}
    If $e\in E$ is neither a loop nor a bridge of the graph $G$, then
    \[
        h^\ast(\cC_G; z) = (1+z)\, h^\ast(\cC_{G\setminus e}; z)+2z\, h^\ast(\cC_{G/e}; z) \enspace .
    \]
    Furthermore, if $e$ is a loop of $G$, then
    $h^\ast(\cC_G; z) = (1+z)\, h^\ast(\cC_{G\setminus e}; z)$, 
    and if  $e$ is a bridge, then $h^\ast(\cC_G; z) = (1+3z)\, h^\ast(\cC_{G/e}; z)$. Moreover, if $E=\emptyset$, then $h^\ast(\cC_G; z) = 1$.
\end{corollary}
\begin{proof}
    The cosmological polytope is a simplex if $E=\emptyset$, and thus $h^\ast(\cC_G; z) = 1$ in this case.
    We apply \cite[Theorem 3.2.]{BruckampGoltermannJuhnkeLandinSolus}
    if $e$ is a loop or a bridge, and the general recurrence follows directly from Theorem~\ref{thm:hstar}.
\end{proof}

Graph invariants that satisfy the deletion-contraction principle are called \emph{(generalized) Tutte-Grothendiek invariants}; see for the details Chapter~2 of \cite{handbook-Tutte}. The bivariate Tutte polynomial
\[
    \Tutte_G(x,y) \ = \ \sum_{H\subseteq E} (x-1)^{c(H)-c(E)} (y-1)^{c(H)+|H|-|V|} \in \ZZ[x,y]
\]
is a universal Grothendiek invariant of the graph $G$ with edges $E$ and nodes $V$, where $c(H)$ denotes the number of connected components of the subgraph $G|_H$. Hence we obtain the $h^\ast$-polynomial of a cosmological polytope as a specialization of the Tutte polynomial.
\begin{theorem}\label{thm:tutte} Let $G$ be a graph with $m$ edges and loops and rank $r=|V|-c(E)$, then
        \[
        h^\ast(\cC_G; z)\ =\ (1+z)^{m-r}\, (2z)^r \Tutte_G\biggl(\frac{1+3z}{2z},1\biggr) \enspace .
    \]
\end{theorem}

Not all of these $h^\ast$-polynomials are real-rooted; see also Remark~\ref{rem:real-rooted} below.
However, Petter Br\"and\'en and June Huh showed in \cite{BrandenHuh:2020} that the homogenized bivariate Tutte polynomial of any matroid and thus graph is Lorentzian.
We draw the following conclusion on the $h^*$-polynomials of any cosmological polytope from their finding.

\begin{corollary}\label{cor:log-concave}
    The coefficients $h^*_i$ of the $h^*$-polynomial
    $h^\ast(\cC_G; z)\ =\ \sum_{i=0}^m  h^*_i z^i$ of the cosmological polytope $\cC_G$
    form an ultra log-concave sequence, i.e., \[i(m-i)(h_i^*)^2 \geq (i+1)(m-i+1)h_{i-1}^*\,h_{i+1}^*.\]
\end{corollary}
\begin{proof}
    This follows directly from the formula on \cite[page 825] {BrandenHuh:2020} applied to the dual of the cycle matroid of $G$ with $w = \frac{1+z}{2z}$, $q=0$ and multiplied by $(2z)^m$.
\end{proof}

Theorem~\ref{thm:hstar} tells us more about the coefficients of any $h^*$-polynomial of a cosmological polytope. These coefficients are maximal whenever the graph $G$ is a simple tree.  Therefore we get a positive answer to \cite[Conjecture 3.8]{BruckampGoltermannJuhnkeLandinSolus}.
\begin{corollary}\label{cor:bound}
    Let $G$ be a graph with $m$ edges and loops. Then the $i$-the coefficient $h^*_i$ of $h^*(\cC_G,z)$ is non-negative and bounded from above by $3^i\tbinom{m}{i}$.
\end{corollary}

\begin{proof}
    The number of terms in the formula of Theorem~\ref{thm:hstar} is maximal whenever $G$ is a tree.
    In this case
    \begin{align*}
        h^*(\cC_G,z) &= \sum_{H\subseteq E} (2z)^{|H|} (1+z)^{m-|H|}
                     = \sum_{k=0}^m \binom{m}{k} (2z)^k(1+z)^{m-k}
                     = (1+3z)^m
    \end{align*}
showing that $h^*_i\leq 3^i\tbinom{m}{i}$ for any graph.
\end{proof}

\begin{remark}
    We also want to point out that some graphs have the same Tutte polynomial and thus the $h^\ast$-polynomials of their cosmological polytopes are the same. However, these cosmological polytopes might not be unimodular equivalent.
    A concrete example is the cosmological polytope of a path with at least three edges and the cosmological polytope of a star graph with the same number of edges.
    They have the same $h^\ast$-polynomial but are not unimodular equivalent as the latter graph has a vertex $v$ of degree more than two. Thus in the latter polytope lies the point $e_v$ in more than three lines of vertices.
    A pair of $2$-connected graphs with a similar property can be found in \cite[Figure 5.9]{Oxley:2011}.
\end{remark}

Now we take a closer look at four families of graphs whose cosmological polytopes have been studied previously.
The first example are multitrees which generalizes the trees we considered for Corollary~\ref{cor:bound}.
\begin{example}[Multitrees]\label{ex:multitrees} Let $T$ be a multitree with $n$ nodes and $E_1, \ldots E_{n-1}$ classes of parallel edges such that 
$\sum_{i=1}^{n-1} |E_i| = m$.
We denote by $a_i$ the number of edges in $E_i$. Then
by Theorem~\ref{thm:hstar} the $h^\ast$-polynoimal of 
$\cC_T$ is
\begin{align*}
    h^\ast(\cC_T;z) \ = \ \sum_{I\subseteq [n-1]} \Bigl( \prod_{i\in I} a_i\Bigr) (2z)^{|I|} (1+z)^{m-|I|}
    \ = \ \prod_{i=1}^{n-1} \Bigl( 2 a_i z (1+z)^{a_i-1} + (1+z)^{a_i} \Bigr)
\end{align*}
where $I$ indexes for which classes of parallel edges a double edge is selected. If $i\in I$ then there are $a_i$ ways to select the double edge in $E_i$. This formula agrees with \cite[Theorem~4.9]{BruckampGoltermannJuhnkeLandinSolus}.
\end{example}

A second family of graphs that had been studied are multicycles. Here we make use of the inclusion-exclusion principle.
\begin{example}[Multicycles]\label{ex:multicycles} Let $C$ be a multicycle with $n$ nodes and $E_1, \ldots E_n$ classes of parallel edges.
such that, as in the previous example, we denote by $a_i$ the number of edges in $E_i$ and $\sum a_i = m$, then
by Theorem~\ref{thm:hstar} and the calculation in Example~\ref{ex:multitrees}
\begin{align*}
    h^\ast(\cC_C;z) \ = \ \prod_{i=1}^{n} \Bigl( 2 a_i z (1+z)^{a_i-1} + (1+z)^{a_i} \Bigr) 
    - \prod_{i=1}^{n}  2 a_i z (1+z)^{a_i-1} \enspace .
\end{align*}
Where the minuend is obtained by allowing for all subsets of edges and the subtrahend is the correction that excludes those graphs for which the double edges from a cycle. Also in this situation the formula was known previously and it agrees with \cite[Theorem~4.6]{BruckampGoltermannJuhnkeLandinSolus}.
\end{example}

\begin{remark}\label{rem:real-rooted}
    The above example shows that the $h^\ast$-polynomial of a simple three-cycle is $(1+3z)^3-(2z)^3 = (z+1) (19z^2+8z+1)$ and thus not real-rooted. 
\end{remark}

Before we look at the next family of graphs we want to reformulate the formula of Theorem~\ref{thm:hstar} by applying the inclusion-exclusion principle we just saw in Example~\ref{ex:multicycles}.
For this purpose it is convenient to recall the next standard definition from enumerative combinatorics.
\begin{definition} Let $P$ be a poset with minimal element $\emptyset$ which is \emph{locally finite}, i.e, every interval of $P$ is finite. The \emph{M\"obius function} on a locally fintite poset $P$ is the unique function $\mu:P\times P \to \ZZ$  satisfying the following three properties.
\begin{itemize}
    \item $\mu(s,s) = 1$
    \item $\mu(s,t) = 0$ whenever $s\not\leq t$
    \item and for all $s\leq t$ holds
    \[
        \sum_{s\leq r \leq t} \mu(s,r) = 0
    \]
\end{itemize}
We write $\mu(s)$ for $\mu(\emptyset, s)$.
\end{definition}
The interested reader may consult Richard Stanley's book \cite{Stanley:2012} or Martin Aigner's book \cite{Aigner:1997} for further properties and applications of Möbius functions of lattices and posets, particularly the M\"obius inversion formula, which we are going to apply for our next result.

\begin{theorem}\label{thm:hstar2}
    Let $G$ be a multigraph with $m$ edges and loops,
    then
    \[
        h^\ast(\mathcal{C}_G;z)\ =\ \sum_{H} \mu(H)\, (2z)^{|H|}\,(1+3z)^{m-|H|}
    \]
    where the sum is taken over all subsets $H\subseteq E$ of $G$ such that $G|_H$ has no bridge, and $\mu$ is the Möbius function on the poset formed by these sets ordered by inclusion. 
\end{theorem}
\begin{proof}
    We define two functions $f,g$ from the poset of unions of cycles of $G=(V,E)$ to $\ZZ[z]$ which are related via  M\"obius inversion.
    Let 
    \[
        f(T)\ =\ \sum_H (2z)^{|H|}(1+z)^{m-|H|}
    \]
    where the sum taken over all sets $H\subseteq E$ that contain $T$ and no cycle of $G$ that is not already in $T$.
    In particular, $f(\emptyset) = h^\ast(\mathcal{C}_G;z)$.
    Furthermore, let
    \[
        g(S)\ =\ \sum_T f(T)
    \]
    where the sum is taken over all elements $T$ in the poset that contain $S$, thus $g(S) = (2z)^{|S|}(1+3z)^{m-|S|}$ which follows from the binomial theorem with an additional factor $2^{|S|}$. Now the M\"obius inversion formula implies that
    \[
    f(S)\ =\ \sum_H \mu(S,H) (2z)^{|H|}(1+3z)^{m-|H|}
    \]
    where the sum is again taken over all elements $H$ in the poset that contain $S$. For $S=\emptyset$ we obtain the desired formula of the $h^\ast$-polynomial.
\end{proof}

Now we are prepared to take a closer look at $\Theta$-graphs. A \emph{$\Theta$-graph} is the union of three edge-disjoint paths connecting two (distinct) nodes.
We denote by $\Theta_{a,b,c}$ the $\Theta$-graph whose paths are of length $a$, $b$ and $c$, respectively.
That is $\Theta_{a,b,c}$ has $m=a+b+c$ edges, $n=a+b+c-1$ nodes and cycles of length $a+b$, $a+c$ and $b+c$. 
The simplest example is $\Theta_{1,1,1}$ the multitree with two nodes and three parallel edges.

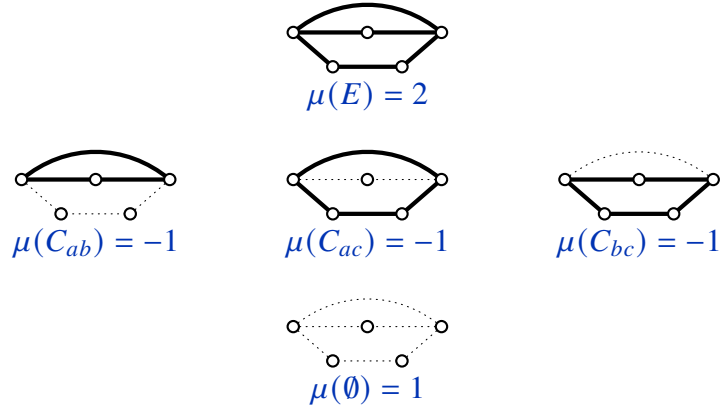
\begin{figure}[t]
  \centering
    \input{thetagraph.tikz}
  \caption{The poset of unions of cycles of the graph $\Theta_{a,b,c}$ with $a=1$, $b=2$ and $c=3$ and its five elements $\emptyset$, $C_{ab}$, $C_{ac}$, $C_{bc}$ and $E=C_{ab}\cup C_{ac}\cup C_{bc}$. }
  \label{fig:theta}
\end{figure}

\begin{example}[$\Theta$-graphs]\label{ex:theta} The graph $\Theta_{a,b,c}$ has three cycles and the union of any two of them covers the entire graph, thus the poset in Theorem~\ref{thm:hstar2} consists of five elements; see also Figure~\ref{fig:theta}.
Applying the theorem yields that the $h^\ast$-polynomial of the cosmological polytope of $\Theta_{a,b,c}$ is
\[
(1+3z)^{a+b+c}-(2z)^{a+b}(1+3z)^c-(2z)^{a+c}(1+3z)^b-(2z)^{b+c}(1+3z)^a+2(2z)^{a+b+c}\enspace ,
\]
because $\mu(\emptyset) = 1$, $\mu(C) = -1$ for each cycle and $\mu(E) = 2$. 
\end{example}
The previous example gives a positive answer to \cite[Conjecture 4.14.]{BruckampGoltermannJuhnkeLandinSolus}.

Our last example are the complete bipartite graphs $K_{2,n}$ whose cosmological polytopes were the content of Erik Landin's master thesis \cite{Landin:2023}.
\begin{example}[Some Bipartite Graphs]\label{ex:bipartite} The complete bipartite graph $K_{2,n}$ has $2n$ edges and $n+2$ nodes. The poset generated by its cycles is formed by the elements of the Boolean lattice on $n$ elements of size at least two and the empty set as they describe the selected paths. The M\"obius function of this poset is the alternating function $(-1)^j (j-1)$ if $j\geq 2$ paths are selected. Thus Theorem~\ref{thm:hstar2} leads to the following 
\begin{align*}
h^\ast(\cC_{K_{2,n}};z) &= (1+3z)^{2n} - \sum_{j=2}^n (-1)^j (j-1) \binom{n}{j} (2z)^{2j}(1+3z)^{2n-2j}\\
&= \Bigl(1+6z+5z^2\Bigr)^{n} + 4nz^2\Bigl(1+6z+5z^2\Bigr)^{n-1}
\enspace .
\end{align*}
Evaluating this polynomial at $z=1$ results in Landin's volume formula $(1+\tfrac{n}{3})12^{n}$.
\end{example}

The original motivation to study the $h^\ast$-polynomials of cosmological polytopes was to get a better understanding of their volume to estimate the complexity of the computation of wavefunctions. Thus we close with a combinatorial formula for the volume. This formula is a direct consequence of Theorem~\ref{thm:hstar} and Theorem~\ref{thm:tutte}.

\begin{corollary}\label{cor:volume}
    The (normalized) volume of the cosmological polytope $\cC_G$ is given by
    \[
        h^\ast(\cC_G; 1) \ = \ 2^m\, \Tutte_G(2,1)
    \]
    where $\Tutte_G$ denotes the Tutte polynomial of $G$ and hence $\Tutte_G(2,1)$ is the number of acyclic edge subsets of the graph $G$.
\end{corollary}

\begin{remark}
    Combining the upper bound of Corollary~\ref{cor:bound} with the volume formula of Corollary~\ref{cor:volume} 
    reveals that the volume or the number of unimodular simplices in a triangulation of  a cosmological polytope $\cC_G$ of a graph with $m$ edges is at most $4^m$. This bound is achieved if and only if the graph $G$ is a forest. The minimal number of simplices in a unimodular triangulation is $2^m$; see also \cite[Theorem 3.5]{BruckampGoltermannJuhnkeLandinSolus}, this is the case whenever all edges of the graph $G$ are loops and hence $\Tutte_G(2,1)=1$.
\end{remark}

\section*{Acknowledgments} 
\noindent The authors would like to thank various colleagues for fruitful discussions and comments that helped us to improve this article.
 In particular, Liam Solus and Katharina Jochemko.
 We also want to thank Raman Sanyal for his support of the project.

\bibliographystyle{alpha}
\bibliography{bibliography}

\end{document}

%% file: example1dec.tikz
\begin{tikzpicture}[scale = 0.9,
                    color = {black}]

  \tikzset{->-/.style={decoration={ markings, mark=at position #1 with {\arrow{>}}},postaction={decorate}}}
  
  \definecolor{myblue}{rgb}{0,0.2,0.7}
  \tikzstyle{linestyle} = [ultra thick, line cap=round, line join=round];
  \tikzstyle{linestyle2}= [decorate, decoration=snake, ultra thick, line cap=round, line join=round]
  \tikzstyle{linestyle3} = [->-=0.8,> = stealth, ultra thick, line cap=round, line join=round];
  \tikzstyle{linestyleX} = [dotted, line cap=round, line join=round];
  \tikzstyle{nodestyle} = [draw=black, thick, circle, fill=white, inner sep=1.5pt];
  \tikzstyle{nodestyle2} = [draw=black, thick, circle, fill=black, inner sep=1.5pt];
  \tikzstyle{nodestyle3} = [draw=none, fill=white, circle, inner sep=0pt];

  \draw[ultra thick, white] (-3.45,-3.45,0) -- (3.45,3.45,0);
  
  \newcommand{\graphx}[4]{
    \coordinate (t) at (#1,#2);
    \coordinate (u) at ($(-.5, 0)+(t)$);
    \coordinate (v) at ($(.5, 0)+(t)$);

    \ifthenelse{#3=0}{
        \node[nodestyle3] at ($(0, .4)+(t)$) {\tiny $f$};
        \draw[linestyle] (u) to[bend left=20] (v);
    }{
    \ifthenelse{#3=1}{
        \node[nodestyle3] at ($(0, .5)+(t)$) {\tiny $\widetilde{f}$};
        \draw[linestyle2] (u) to[bend left=20] (v);
    }{
    \ifthenelse{#3=2}{
        \node[nodestyle3] at ($(0, .6)+(t)$) {\tiny $f$, $\overleftarrow{f}$};
        \draw[linestyle3] (v) to[bend right=45] (u);
        \draw[linestyle] (v) to[bend right=10] (u);
    }{
        \node[nodestyle3] at ($(0, .6)+(t)$) {\tiny $f$, $\overrightarrow{f}$};
        \draw[linestyle3] (u) to[bend left=45] (v);
        \draw[linestyle] (u) to[bend left=10] (v);
    }
    }
    }
    \ifthenelse{#4=0}{
        \node[nodestyle3] at ($(0, -.4)+(t)$) {\tiny $f'$};
        \draw[linestyle] (u) to[bend right=20] (v);
    }{
    \ifthenelse{#4=1}{
        \node[nodestyle3] at ($(0,-.5)+(t)$) {\tiny $\widetilde{f}'$};
        \draw[linestyle2] (u) to[bend right=30] (v);
    }{
    \ifthenelse{#4=2}{
        \node[nodestyle3] at ($(0,-.6)+(t)$) {\tiny $f'$, $\overleftarrow{f}'$};
        \draw[linestyle3] (v) to[bend left=45] (u);
        \draw[linestyle] (v) to[bend left=10] (u);
    }{
    \ifthenelse{#4=3}{
    \node[nodestyle3] at ($(0,-.6)+(t)$) {\tiny $f'$, $\overrightarrow{f}'$};
        \draw[linestyle3] (u) to[bend right=45] (v);
        \draw[linestyle] (u) to[bend right=10] (v);
    }{
    \ifthenelse{#3=3}{
        \node[nodestyle3] at ($(0,-.6)+(t)$) {\tiny $\overrightarrow{f}'$};
        \draw[linestyle3] (u) to[bend right=45] (v);
        \draw[linestyleX] (u) to[bend right=12] (v);
    }{
        \node[nodestyle3] at ($(0,-.6)+(t)$) {\tiny $\overleftarrow{f}'$};
        \draw[linestyle3] (v) to[bend left=45] (u);
        \draw[linestyleX] (u) to[bend right=12] (v);
    }
    }
    }
    }
    }
    \ifthenelse{#3=2 \or #4=2}{
        \node[nodestyle3] at ($(-.2, -.1)+(u)$) {\tiny $u$};
        \node[nodestyle2] at (u) {};
        \node[nodestyle] at (v) {};
    }{
    \ifthenelse{#3=3 \or #4=3}{
        \node[nodestyle3] at ($(.2, -.1)+(v)$) {\tiny $v$};
        \node[nodestyle] at (u) {};
        \node[nodestyle2] at (v) {};
    }{
        \node[nodestyle3] at ($(-.2, -.1)+(u)$) {\tiny $u$};
        \node[nodestyle3] at ($(.2, -.1)+(v)$) {\tiny $v$};
        \node[nodestyle2] at (u) {};
        \node[nodestyle2] at (v) {};
    }
    }

  }

  \graphx{-3}{ 2}{0}{0}
  \graphx{-1}{ 2}{1}{0}
  \graphx{ 1}{ 2}{2}{4} 
  \graphx{ 3}{ 2}{3}{4} 
  
  \graphx{-3}{ 0}{0}{1}
  \graphx{-1}{ 0}{1}{1}
  \graphx{ 1}{ 0}{2}{1}
  \graphx{ 3}{ 0}{3}{1}
  
  \graphx{-3}{-2}{0}{2}
  \graphx{-1}{-2}{1}{2}
  \graphx{ 1}{-2}{0}{3}
  \graphx{ 3}{-2}{1}{3}

\end{tikzpicture}

%% file: example2dec.tikz
\begin{tikzpicture}[scale = 0.7,
                    color = {black}]

  \tikzset{->-/.style={decoration={ markings, mark=at position #1 with {\arrow{>}}},postaction={decorate}}}
  
  \definecolor{myblue}{rgb}{0,0.2,0.7}
  \tikzstyle{linestyle} = [ultra thick, line cap=round, line join=round];
  \tikzstyle{linestyle2}= [decorate, decoration=snake, ultra thick, line cap=round, line join=round]
  \tikzstyle{linestyle3} = [->-=0.8,> = stealth, ultra thick, line cap=round, line join=round];
  \tikzstyle{linestyleX} = [dotted, line cap=round, line join=round];
  \tikzstyle{nodestyle} = [draw=black, thick, circle, fill=white, inner sep=1.5pt];
  \tikzstyle{nodestyle2} = [draw=black, thick, circle, fill=black, inner sep=1.5pt];
  \tikzstyle{nodestyle3} = [draw=none, fill=white, circle, inner sep=0pt];
  \tikzstyle{nodestyleT} = [draw=none, fill=white, circle, inner sep=0pt];

  \draw[ultra thick, white] (-3,-3) -- (3,3);
  
  \newcommand{\graphx}[5]{
    \coordinate (t) at (#1,#2);
    \coordinate (u) at ($(-1, 0)+(t)$);
    \coordinate (v) at ($(0, 0)+(t)$);
    \coordinate (w) at ($(1, 0)+(t)$);

    \ifthenelse{#3=0}{
        \draw[linestyle] (u) to (v);
    }{
    \ifthenelse{#3=1}{
        \draw[linestyle2] (u) to (v);
    }{
    \ifthenelse{#3=2}{
        \draw[linestyle3] (v) to[bend right=45] (u);
        \draw[linestyle] (v) to (u);
    }{
    \ifthenelse{#3=4}{
        \draw[linestyle3] (v) to[bend right=45] (u);
        \draw[linestyleX] (u) to (v);
    }{
        \draw[linestyle3] (u) to[bend left=45] (v);
        \draw[linestyle] (u) to (v);
    }
    }
    }
    }
    \ifthenelse{#4=0}{
        \draw[linestyle] (v) to (w);
    }{
    \ifthenelse{#4=1}{
        \draw[linestyle2] (v) to (w);
    }{
    \ifthenelse{#4=2}{
        \draw[linestyle3] (w) to[bend right=45] (v);
        \draw[linestyle] (v) to (w);
    }{
        \draw[linestyle3] (v) to[bend left=45] (w);
        \draw[linestyle] (v) to (w);
    }
    }
    }
    \ifthenelse{#3=2}{
        \ifthenelse{#4=3}{
        \node[nodestyle] at (u) {};
        }{
        \node[nodestyle2] at (u) {};
        }
        \node[nodestyle] at (v) {};
    }{
    \ifthenelse{#3=3}{
        \ifthenelse{#4=3}{
        }{
        \node[nodestyle2] at (v) {};
        }
        \node[nodestyle] at (u) {};
    }{
        \node[nodestyle2] at (u) {};
    }
    }
    \ifthenelse{#4=2}{
        \ifthenelse{#3=0 \or #3=1}{
            \node[nodestyle2] at (v) {};
        }{}
        \node[nodestyle] at (w) {};
    }{
    \ifthenelse{#4=3}{
        \node[nodestyle] at (v) {};
        \node[nodestyle2] at (w) {};
    }{
        \ifthenelse{#3=0 \or #3=1}{
            \node[nodestyle2] at (v) {};
        }{}
        \node[nodestyle2] at (w) {};
    }
    }

    \node[nodestyleT] at ($(-1.3, 0.5)+(t)$) {\tiny #5)};
  }

  \graphx{-3.9}{1.5}{0}{0}{1}
  \graphx{-1.3}{1.5}{1}{0}{2}
  \graphx{ 1.3}{1.5}{2}{0}{3} 
  \graphx{ 3.9}{1.5}{3}{0}{4} 

  \graphx{-3.9}{ 0}{0}{1}{5}
  \graphx{-1.3}{ 0}{1}{1}{6}
  \graphx{ 1.3}{ 0}{2}{1}{7}
  \graphx{ 3.9}{ 0}{3}{1}{8}
  
  \graphx{-3.9}{-1.5}{0}{2}{9}
  \graphx{-1.3}{-1.5}{1}{2}{10}
  \graphx{ 1.3}{-1.5}{2}{2}{11}
  \graphx{ 3.9}{-1.5}{3}{2}{12}

  \graphx{-3.9}{-3}{4}{3}{13}
  \graphx{-1.4}{-3}{1}{3}{14}
  \graphx{ 1.3}{-3}{2}{3}{15}
  \graphx{ 3.9}{-3}{3}{3}{16}

\end{tikzpicture}

%% file: example2dualgraph.tikz
\begin{tikzpicture}[x  = {(1cm,-.2cm)},
                    y  = {(0.3cm,1cm)},
                    z  = {(0cm,1cm)},
                    scale = 0.8,
                    label distance=2,
                    color = {black}]

  \tikzset{->-/.style={decoration={ markings, mark=at position #1 with {\arrow{>}}},postaction={decorate}}}
  
  \definecolor{myblue}{rgb}{0,0.2,0.7}
  
  \tikzstyle{preaction} = [draw=white, line cap=round,  line join=round, line width=2.7pt];
  \tikzstyle{linestyle} = [->-=0.86,> = stealth, ultra thick, line cap=round, line join=round, fill=none, draw=black, line width=1.5 pt];
  \tikzstyle{nodestyle} = [draw=black, thick, circle, fill=black, inner sep=1.5pt];
  \tikzstyle{nodestyleT} = [draw=none, fill=none, inner sep=0pt];

  \draw[ultra thick, white] (-3,-3,0) -- (2.5,3,0);
  
  \coordinate (v0) at (   0,    0,   0);
  \coordinate (v1) at (-1.2, -1.2,   0);
  \coordinate (v2) at (   0,    0, 1.4);
  \coordinate (v3) at ( 1.2,  1.2,   0);
  \coordinate (v4) at (-1.2,  1.2,   0);
  \coordinate (v5) at ($(v1)+(v4)$);
  \coordinate (v6) at ($(v2)+(v4)$);
  \coordinate (v7) at ($(v3)+(v4)$);
  \coordinate (v8) at ( 1.2, -1.2,   0);
  \coordinate (v9) at ($(v1)+(v8)$);
  \coordinate (v10) at ($(v2)+(v8)$);
  \coordinate (v11) at ($(v3)+(v8)$);
  \coordinate (v13) at ($2.2*(v2)+0.6*(v1)$);
  \coordinate (v12) at ($(v1)+(v13)-(v2)$);
  \coordinate (v14) at ($1.8*(v2)+0.6*(v3)$);
  \coordinate (v15) at ($(v3)+(v14)-(v2)$);

  \draw[linestyle] (v0) -- (v1);
  \draw[linestyle] (v0) -- (v3);
  \draw[linestyle] (v0) -- (v4);
  \draw[linestyle] (v0) -- (v8);
  \draw[linestyle] (v3) -- (v7);
  \draw[linestyle] (v4) -- (v7);
  \draw[linestyle] (v4) -- (v5);
  \draw[linestyle] (v1) -- (v5);
  \draw[linestyle] (v3) -- (v11);
  \draw[linestyle] (v8) -- (v11);
  \draw[linestyle] (v1) -- (v9);
  \draw[linestyle] (v8) -- (v9);

  \draw[linestyle] (v4) -- (v6);
  \draw[linestyle] (v0) -- (v2);
  \draw[linestyle] (v8) -- (v10);
  \draw[preaction] (v2) -- (v6);
  \draw[preaction] (v2) -- (v10);
  \draw[linestyle] (v2) -- (v6);
  \draw[linestyle] (v2) -- (v10);
  
  \draw[preaction] (v3) -- (v15);
  \draw[linestyle] (v3) -- (v15);
  \draw[preaction] (v14) -- (v15);
  \draw[linestyle] (v14) -- (v15);
  \draw[preaction] (v2) -- (v14);
  \draw[linestyle] (v2) -- (v14);
  
  \draw[preaction] (v2) -- (v13);
  \draw[linestyle] (v2) -- (v13);
  \draw[preaction] (v1) -- (v12);
  \draw[linestyle] (v1) -- (v12);
  \draw[preaction] (v13) -- (v12);
  \draw[linestyle] (v13) -- (v12);
  
  \draw[preaction] (v13) -- (v14);
  \draw[linestyle] (v13) -- (v14);
  
  \node[nodestyleT] at ($(v0)-(0, .3,0)$) {\tiny $1$};
  \node[nodestyleT] at ($(v1)-(0, .3,0)$) {\tiny $2$};
  \node[nodestyleT] at ($(v2)-(-0.3,  0,0)$) {\tiny $3$};
  \node[nodestyleT] at ($(v3)-(0, .3,0)$) {\tiny $4$};
  \node[nodestyleT] at ($(v4)-(0, .3,0)$) {\tiny $5$};
  \node[nodestyleT] at ($(v5)-(0.3, 0,0)$) {\tiny $6$};
  \node[nodestyleT] at ($(v6)-(0,-.3,0)$) {\tiny $7$};
  \node[nodestyleT] at ($(v7)-(0, .3,0)$) {\tiny $8$};
  \node[nodestyleT] at ($(v8)-(0, .3,0)$) {\tiny $9$};
  \node[nodestyleT] at ($(v9)-(0, .3,0)$) {\tiny $10$};
  \node[nodestyleT] at ($(v10)-(-0.3, 0,0)$) {\tiny $11$};
  \node[nodestyleT] at ($(v11)-(0, .3,0)$) {\tiny $12$};
  \node[nodestyleT] at ($(v12)-(0, .3,0)$) {\tiny $14$};
  \node[nodestyleT] at ($(v13)-(0.3, 0,0)$) {\tiny $13$};
  \node[nodestyleT] at ($(v14)-(0,-.3,0)$) {\tiny $15$};
  \node[nodestyleT] at ($(v15)-(0,-.3,0)$) {\tiny $16$};

  \node[nodestyle] at (v0) {};
  \node[nodestyle] at (v1) {};
  \node[nodestyle] at (v2) {};
  \node[nodestyle] at (v3) {};
  \node[nodestyle] at (v4) {};
  \node[nodestyle] at (v5) {};
  \node[nodestyle] at (v6) {};
  \node[nodestyle] at (v7) {};
  \node[nodestyle] at (v8) {};
  \node[nodestyle] at (v9) {}; 
  \node[nodestyle] at (v10) {}; 
  \node[nodestyle] at (v11) {};
  \node[nodestyle] at (v12) {}; 
  \node[nodestyle] at (v13) {}; 
  \node[nodestyle] at (v14) {}; 
  \node[nodestyle] at (v15) {}; 

\end{tikzpicture}

%% file: thetagraph.tikz
\begin{tikzpicture}[scale = 0.65,
                    color = {black}]
  
  \definecolor{myblue}{rgb}{0,0.2,0.7}
  \tikzstyle{linestyle} = [ultra thick, line cap=round, line join=round];
  \tikzstyle{linestyle2} = [dotted, line cap=round, line join=round];
  \tikzstyle{nodestyle} = [draw=black, thick, circle, fill=white, inner sep=1.5pt];

  \coordinate (tu) at (0, 0);
  \coordinate (tv) at (3, 0);
  \coordinate (tm) at (1.5, 0);
  \coordinate (tl1) at (0.8, -.7);
  \coordinate (tl2) at (2.2, -.7);
   \coordinate (mu) at (1.5, -1.3);

  \draw[linestyle] (tu) to[bend left=40] (tv);
  \draw[linestyle] (tu) -- (tm) -- (tv);
  \draw[linestyle] (tu) -- (tl1) -- (tl2) -- (tv);
   \node[nodestyle] at (tu) {};
   \node[nodestyle] at (tv) {};
   \node[nodestyle] at (tm) {};
   \node[nodestyle] at (tl1) {};
   \node[nodestyle] at (tl2) {};
   \node[myblue,inner sep=1pt] at (mu) {$\mu(E) = 2$};

  \coordinate (t) at (-5.5,-3);
  \coordinate (Lu) at ($(tu)+(t)$);
  \coordinate (Lv) at ($(tv)+(t)$);
  \coordinate (Lm) at ($(tm)+(t)$);
  \coordinate (Ll1) at ($(tl1)+(t)$);
  \coordinate (Ll2) at ($(tl2)+(t)$);

   \draw[linestyle] (Lu) to[bend left=40] (Lv);
   \draw[linestyle] (Lu) -- (Lm) -- (Lv);
   \draw[linestyle2] (Lu) -- (Ll1) -- (Ll2) -- (Lv);
   \node[nodestyle] at (Lu) {};
   \node[nodestyle] at (Lv) {};
   \node[nodestyle] at (Lm) {};
   \node[nodestyle] at (Ll1) {};
   \node[nodestyle] at (Ll2) {};
   \node[myblue,inner sep=1pt] at ($(mu)+(t)$) {$\mu(C_{ab}) = -1$};
  
  \coordinate (t2) at (0,-3);
  \coordinate (Mu) at ($(tu)+(t2)$);
  \coordinate (Mv) at ($(tv)+(t2)$);
  \coordinate (Mm) at ($(tm)+(t2)$);
  \coordinate (Ml1) at ($(tl1)+(t2)$);
  \coordinate (Ml2) at ($(tl2)+(t2)$);

   \draw[linestyle] (Mu) to[bend left=40] (Mv);
   \draw[linestyle2] (Mu) -- (Mm) -- (Mv);
   \draw[linestyle] (Mu) -- (Ml1) -- (Ml2) -- (Mv);
   \node[nodestyle] at (Mu) {};
   \node[nodestyle] at (Mv) {};
   \node[nodestyle] at (Mm) {};
   \node[nodestyle] at (Ml1) {};
   \node[nodestyle] at (Ml2) {};
   \node[myblue,inner sep=1pt] at ($(mu)+(t2)$) {$\mu(C_{ac}) = -1$};
   
  \coordinate (t3) at (5.5,-3);
  \coordinate (Ru) at ($(tu)+(t3)$);
  \coordinate (Rv) at ($(tv)+(t3)$);
  \coordinate (Rm) at ($(tm)+(t3)$);
  \coordinate (Rl1) at ($(tl1)+(t3)$);
  \coordinate (Rl2) at ($(tl2)+(t3)$);

   \draw[linestyle2] (Ru) to[bend left=40] (Rv);
   \draw[linestyle] (Ru) -- (Rm) -- (Rv);
   \draw[linestyle] (Ru) -- (Rl1) -- (Rl2) -- (Rv);
   \node[nodestyle] at (Ru) {};
   \node[nodestyle] at (Rv) {};
   \node[nodestyle] at (Rm) {};
   \node[nodestyle] at (Rl1) {};
   \node[nodestyle] at (Rl2) {};
   \node[myblue,inner sep=1pt] at ($(mu)+(t3)$) {$\mu(C_{bc}) = -1$};

  \coordinate (t4) at (0,-6);
  \coordinate (du) at ($(tu)+(t4)$);
  \coordinate (dv) at ($(tv)+(t4)$);
  \coordinate (dm) at ($(tm)+(t4)$);
  \coordinate (dl1) at ($(tl1)+(t4)$);
  \coordinate (dl2) at ($(tl2)+(t4)$);

   \draw[linestyle2] (du) to[bend left=40] (dv);
   \draw[linestyle2] (du) -- (dm) -- (dv);
   \draw[linestyle2] (du) -- (dl1) -- (dl2) -- (dv);
   \node[nodestyle] at (du) {};
   \node[nodestyle] at (dv) {};
   \node[nodestyle] at (dm) {};
   \node[nodestyle] at (dl1) {};
   \node[nodestyle] at (dl2) {};
   \node[myblue,inner sep=1pt] at ($(mu)+(t4)$) {$\mu(\emptyset) = 1$};

\end{tikzpicture}